\documentclass[10pt,reqno]{amsart}
\usepackage{amsmath}
\usepackage{amsthm}
\usepackage{amssymb}
\usepackage{amsfonts}
\usepackage{latexsym}
\usepackage{hyperref}

\renewcommand{\phi}{\varphi}
\renewcommand{\Re}[1]{\operatorname{Re} #1 }
\renewcommand{\Im}[1]{\operatorname{Im} #1}

\newcommand{\norm}[1]{\| #1 \|}
\newcommand{\inner}[1]{\langle #1 \rangle}
\newcommand{\h}{\mathcal{H}}

\newcommand{\C}{\mathbb{C}}

\newcommand{\minimatrix}[4]{\begin{pmatrix} \vspace{5pt}#1 & #2 \\ #3 & #4 \end{pmatrix}  }
\newcommand{\megamatrix}[9]{\begin{pmatrix} #1 & #2 & #3 \\ #4 & #5 & #6 \\ #7 & #8 & #9\end{pmatrix}  }

\newcommand{\UECSMTest}{\texttt{UECSMTest}}
\newcommand{\Test}{\texttt{StrongAngleTest}}

\newcommand{\twovector}[2]{\begin{pmatrix} #1\\#2 \end{pmatrix} }
\newcommand{\threevector}[3]{\begin{pmatrix} #1\\[3pt]#2\\[3pt]#3 \end{pmatrix} }

\newtheorem{Corollary}{Corollary}
\newtheorem{Theorem}{Theorem}

\newtheorem{Lemma}{Lemma}
\theoremstyle{definition}
\newtheorem*{Definition}{Definition}
\newtheorem{Example}{Example}

\allowdisplaybreaks

\begin{document}
    \title[Unitary Equivalence to a Complex Symmetric Matrix]{Unitary Equivalence to 
    a Complex Symmetric Matrix:  Geometric Criteria}

    \author{Levon Balayan}
    \address{   Department of Mathematics\\
		Pomona College\\
    	610 North College Avenue\\
	Claremont, California\\
       91711}
    \email{Levon.Balayan@pomona.edu}
%    \urladdr{http://pages.pomona.edu/\textasciitilde sg064747}
	
    \author{Stephan Ramon Garcia}
%    \address{   Department of Mathematics\\
%		Pomona College\\
%    	610 North College Avenue\\
%	Claremont, California\\
%       91711}
    \email{Stephan.Garcia@pomona.edu}
    \urladdr{http://pages.pomona.edu/\textasciitilde sg064747}
	
%    \author{Matthew Hans Jensen}
%    \address{   Department of Mathematics\\
%		Pomona College\\
%    	610 North College Avenue\\
%	Claremont, California\\
%       91711}
%    \email{Matthew.Jensen@pomona.edu}
%    \urladdr{http://pages.pomona.edu/\textasciitilde sg064747}

    
    \thanks{This work partially supported by National Science Foundation Grant DMS-0638789.}
    
    \begin{abstract}
        	We develop several methods, based on the geometric relationship between the eigenspaces
	of a matrix and its adjoint, for determining whether a square 
	matrix having distinct eigenvalues is unitarily equivalent to a complex symmetric matrix. 
	Equivalently, we characterize those matrices having distinct eigenvalues which lie in the unitary
	orbit of the complex symmetric matrices. 	
    \end{abstract}

\maketitle

\section{Introduction}

	Our aim in this note is to develop simple geometric criteria for determining whether
	a given square matrix $T \in M_n(\C)$ is unitarily equivalent to a complex symmetric
	matrix (UECSM).  To be more specific, a \emph{complex symmetric matrix}
	is a square matrix $T$ with complex entries such that $T = T^t$ (the 
	superscript $t$ denotes the transpose operation)	
	and two matrices $A,B \in M_n(\C)$
	are \emph{unitarily equivalent} if there exists a unitary $U \in M_n(\C)$ such that $A = U^*BU$.

	Our primary motivation stems from the emerging theory of complex symmetric operators
	on Hilbert space (see \cite{Chevrot, CRW, CCO,CSOA,CSO2, Gilbreath, Sarason}, for instance).
	To be more specific, we say 
	that a bounded operator $T$ on a separable complex Hilbert space $\h$
	is a \emph{complex symmetric operator} if $T = CT^*C$ for some
	conjugation $C$ (a conjugate-linear, isometric involution) on $\h$.  The terminology stems from the fact that
	the preceding condition is equivalent to insisting that the operator have 
	a complex symmetric matrix representation with respect to some
	orthonormal basis \cite[Sect.~2.4-2.5]{CCO}.  
	
	From the preceding remarks, we see that the problem of determining whether a given matrix is UECSM is equivalent to
	determining whether that matrix represents a complex symmetric operator with respect to some
	orthonormal basis.  From another perspective, we may view our main problem as part of a quest to determine
	the structure of the unitary orbit of the set of all complex symmetric matrices.

	Complicating this endeavor, it is well-known that every $n \times n$ complex matrix is \emph{similar} to
	a complex symmetric matrix \cite[Thm.~4.4.9]{HJ} (see also \cite[Ex.~4]{CSOA} and \cite[Thm.~2.3]{CCO}).
	It follows that similarity invariants, such as the Jordan canonical form,
	are useless when attempting to determine whether a given matrix is UECSM.
	This greatly complicates our work.  
	For instance, one can show that among the matrices
	\begin{equation*}\small
		\megamatrix{0}{7}{0}{0}{1}{2}{0}{0}{6} \quad
		\megamatrix{0}{7}{0}{0}{1}{3}{0}{0}{6} \quad
		\megamatrix{0}{7}{0}{0}{1}{4}{0}{0}{6} \quad
		\megamatrix{0}{7}{0}{0}{1}{5}{0}{0}{6} \quad
		\megamatrix{0}{7}{0}{0}{1}{6}{0}{0}{6},
	\end{equation*}	
	all of which belong to the same similarity class, only the fourth is UECSM.
	In fact, prior to the recent advent of Tener's procedure 
	\UECSMTest{} \cite{Tener}, only a handful of matrices were known to be
	\emph{not} UECSM.
	
	In fact, we are partly motivated by Tener's \texttt{UECSMTest}.
	His procedure is based upon the diagonalization of the selfadjoint components $A$ and $B$ in the
	Cartesian decomposition $T = A + iB$.  Although highly effective,
	it is often difficult to understand with this method, in simple geometric terms, \emph{why} a given matrix is UECSM or not.
	In particular, studying the matrices $A$ and $B$ often gives little insight into the eigenstructure of $T$ itself.
	
	In this note, we proceed along a different route.  We develop a number of procedures,
	based upon a direct examination of the eigenstructure of $T$, for testing whether $T$ is UECSM or not.
	To this end, we require that $T$ has distinct eigenvalues -- a condition that is satisfied
	by all matrices outside of a set of Lebesgue measure zero in $M_n(\C)$.  On the other hand, Tener's
	\UECSMTest{} requires that neither $A$ nor $B$ have a repeated eigenvalue.
	In Section \ref{SectionTener}, we consider several numerical examples and establish that
	neither our test nor \texttt{UECSMTest} subsumes the other.  They should therefore be viewed
	as complimentary procedures.

\section{Preliminary Setup}\label{SectionPreliminaries}

	Let $T$ be a $n \times n$ complex matrix having $n$ \emph{distinct} eigenvalues
	$\lambda_1, \lambda_2, \ldots, \lambda_n$ and let $u_1,u_2,\ldots, u_n$ denote 
	\emph{normalized} eigenvectors of $T$ corresponding to the eigenvalues $\lambda_i$.
	Since 
	\begin{equation*}
		\det(T^* - \overline{\lambda_i}I) 
		= \det[(T - \lambda_i I)^*] 
		= \overline{ \det(T - \lambda_i I)} 
		= 0,
	\end{equation*}
	it follows immediately that $T^*$ has the $n$ \emph{distinct} eigenvalues
	$\overline{\lambda_1}, \overline{\lambda_2}, \ldots, \overline{\lambda_n}$.  Let $v_1,v_2,\ldots, v_n$ denote 
	\emph{normalized} eigenvectors of $T^*$ corresponding to the eigenvalues $\overline{\lambda_i}$.
	Since eigenvectors corresponding to distinct eigenvalues are linearly independent, it follows that both
	$\{u_1,u_2,\ldots,u_n\}$ and $\{v_1,v_2,\ldots,v_n\}$ are bases for $\C^n$.

	Based upon the data 
	\begin{equation}\label{eq-Data}
		\boxed{ u_1,u_2,\ldots, u_n; v_1,v_2,\ldots, v_n,}
	\end{equation}
	we wish to determine if $T$ is unitarily equivalent to a complex symmetric matrix (UECSM).
	Before proceeding, we require a few preliminary lemmas.

	\begin{Lemma}\label{LemmaBiorthogonality}
		Under the hypotheses above 
		we have $\inner{u_i,v_j} = 0$ whenever $i \neq j$ and 
		$\inner{u_i,v_i} \neq 0$.
	\end{Lemma}

	\begin{proof}
		If $i \neq j$, then 
		\begin{equation*}
			\lambda_i \inner{u_i,v_j}
			= \inner{ \lambda_i u_i, v_j} \\
			= \inner{ Tu_i, v_j} \\
			= \inner{u_i, T^* v_j} \\
			= \inner{u_i, \overline{\lambda_j} v_j} \\
			= \lambda_j \inner{u_i,v_j}
		\end{equation*}
		whence $\inner{u_i,v_j} = 0$ since $\lambda_i \neq \lambda_j$.  
		On the other hand, if $\inner{u_i,v_i} = 0$ for some $i$, then by the preceding
		$\inner{u_i,v_j} = 0$ for $j = 1,2,\ldots, n$.  Since $\{v_1,v_2,\ldots,v_n\}$ is a basis
		for $\C^n$, it would follow that $\inner{u_i,x} = 0$ for all $x \in \C^n$ whence $u_i = 0$.
		This contradiction shows that we must have $\inner{u_i,v_i} \neq 0$ for $i= 1,2,\ldots, n$.
	\end{proof}

	The following lemma allows us to easily express any $x \in \C^n$ in terms of 
	the bases $\{u_1,u_2,\ldots, u_n\}$ and $\{v_1,v_2,\ldots, v_n\}$:

	\begin{Lemma}
		The following formulas hold for all $x \in \C^n$:
		\begin{align}
			x &= \sum_{j=1}^n \frac{ \inner{x,u_j} }{ \inner{ v_j, u_j} } v_j, \label{eq-ReconstructV}\\
			x &= \sum_{j=1}^n \frac{ \inner{x,v_j} }{ \inner{ u_j, v_j} } u_j. \label{eq-ReconstructU}
		\end{align}
	\end{Lemma}

	\begin{proof}
		By symmetry, it suffices to prove \eqref{eq-ReconstructU}.  Since $\{u_1,u_2,\ldots,u_n\}$ is a basis for $\C^n$ and since
		the expression \eqref{eq-ReconstructU} is linear in $x$, it suffices to verify \eqref{eq-ReconstructU} 
		for $x = u_1,u_2,\ldots,u_n$.
		Since $\inner{u_i,v_j} = 0$ if $i \neq j$ and $\inner{u_i,v_i} \neq 0$, \eqref{eq-ReconstructU} can be verified immediately
		by setting $x = u_i$.
	\end{proof}

	Lastly, we require a few words about a useful and practical way to view the property of being UECSM.

	\begin{Definition}
		A \emph{conjugation} on $\C^n$ is a conjugate-linear operator $C:\C^n \rightarrow \C^n$
		which is both \emph{involutive} (i.e., $C^2 = I$) and \emph{isometric} 
		(i.e., $\inner{Cx,Cy} = \inner{y,x}$ for all $x,y\in \C^n$).\footnote{In light of the polarization identity, this
		is equivalent to $\norm{Cx} =\norm{x}$ for all $x \in \C^n$.}
	\end{Definition}

	In particular, $T$ is a complex symmetric matrix if and only if 
	$T$ is $J$-symmetric (i.e., $T = JT^*J$), where $J$ denotes the \emph{canonical conjugation}
	\begin{equation}\label{eq-Canonical}
		J( z_1,z_2, \ldots, z_n) = ( \overline{z_1} , \overline{z_2}, \ldots, \overline{z_n})
	\end{equation}
	on $\C^n$.  Moreover, the most general conjugation on $\C^n$ is easily seen
	to be of the form $C = SJ$ where $S$ is a complex symmetric unitary matrix.
	Lastly, it is not hard to show that $T$ is UECSM if and only if $T$ is $C$-symmetric
	with respect to some conjugation $C$.

\section{The angle test and its relatives}\label{SectionAngle}

	In this section we briefly outline several convenient geometric conditions which are necessary for
	a given $n \times n$ matrix $T$ to be UECSM (unfortunately, none of these procedures
	is sufficient -- see Example \ref{ExampleCounter}).  Building upon this material,
	we present a condition in Section \ref{SectionStrong} which is both necessary \emph{and} sufficient.  

	Recall that $T$ is UECSM if and only
	if there exists a conjugation $C$ on $\C^n$ such that $T = CT^*C$.  If this holds, 
	then it follows easily that 
	\begin{equation}\label{eq-Symmetry}
		(T - \lambda I)^j x = 0  \quad \Leftrightarrow \quad (T^* - \overline{\lambda} I)^j(Cx) =0.
	\end{equation}
	Maintaining the notation and conventions of Section \ref{SectionPreliminaries}, we see that if $T$
	is $C$-symmetric, then
	the conjugation $C$ maps the one-dimensional eigenspace of $T$ corresponding to $\lambda_i$
	onto the one-dimensional eigenspace of $T^*$ corresponding to $\overline{\lambda_i}$.  This
	is where we invoke the hypothesis that the eigenvalues of $T$ are distinct.
	Since $C$ is isometric and the vectors
	$u_i$ and $v_i$ are normalized, it follows that there are \emph{unimodular} constants $\alpha_i$ such that
	\begin{equation*}
		Cu_i = \alpha_i v_i
	\end{equation*}
	for $i = 1,2,\ldots, n$.  Since $C$ is isometric, this implies that
	\begin{align}
		\inner{u_i,u_j}
		&= \inner{Cu_j,Cu_i} \nonumber \\
		&= \inner{ \alpha_j v_j, \alpha_i v_i} \nonumber \\
		&= \alpha_j \overline{ \alpha_i} \inner{ v_j, v_i}  \label{eq-AlphaDefinition}
	\end{align}
	for $1 \leq i,j, \leq n$.  Taking absolute values in the preceding and utilizing symmetry yields the following
	test which can be implemented easily in \texttt{Mathematica}:

	\begin{Theorem}[Angle Test]\label{TheoremWeak}
		Suppose that
		\begin{enumerate}\addtolength{\itemsep}{0.5\baselineskip}
			\item $T$ is a $n \times n$ matrix with distinct eigenvalues $\lambda_1, \lambda_2, \ldots, \lambda_n$,
			\item $u_1,u_2,\ldots, u_n$ denote normalized eigenvectors of $T$ 
				corresponding to the eigenvalues $\lambda_i$,
			\item $v_1,v_2,\ldots, v_n$ denote normalized eigenvectors of $T^*$ 
				corresponding to the eigenvalues $\overline{\lambda_i}$.	
		\end{enumerate}
		Under these hypotheses, the condition $|\inner{u_i,u_j}| = |\inner{v_i,v_j}|$ for all $1 \leq i < j \leq n$ is necessary for
		$T$ to be UECSM.
	\end{Theorem}
	
	In light of the fact that Theorem \ref{TheoremWeak} takes into consideration the
	(complex) angles between the eigenspaces of $T$ and compares them to the (complex)
	angles between the eigenspaces of $T^*$, we refer to the procedure introduced in Theorem \ref{TheoremWeak}
	as the \emph{Angle Test}.
	One can interpret the Angle Test as asserting that the geometric relationship between the eigenspaces
	of $T$ must precisely mirror the geometric relationship between the 
	eigenspaces of $T^*$.  In some sense, $T$ and $T^*$ must be perfect mirror images of each other.
	In Section \ref{SectionStrong}, we present a refined version of 
	Theorem \ref{TheoremWeak} which yields a necessary and sufficient condition for $T$ to be UECSM.
	
	It turns out that the same principles can also be used in certain cases where the eigenvalues
	of $T$ are not distinct.   For instance in \cite[Ex.~7]{CSOA}, a similar argument is used to show that the matrix
	\begin{equation*}
		\megamatrix{1}{a}{0}{0}{0}{b}{0}{0}{1}
	\end{equation*}
	is not UECSM whenever $|a| \neq |b|$.

	The condition \eqref{eq-AlphaDefinition} can also be interpreted in terms of Gram matrices.
	Let $U = (u_1|u_2|\cdots|u_n)$ and $V= (v_1|v_2|\cdots|v_n)$ and
	observe that \eqref{eq-AlphaDefinition} is equivalent to asserting that
	\begin{equation}\label{eq-Grammian}
		(U^*U)^t = A^*(V^*V)A
	\end{equation}
	holds where $A = \operatorname{diag}(\alpha_1,\alpha_2,\ldots,\alpha_n)$ denotes 
	the diagonal unitary matrix having the unimodular constants $\alpha_1,\alpha_2,\ldots,\alpha_n$
	along the main diagonal.  This leads us to the following test:
	
	\begin{Corollary}[Grammian Test]\label{CorollaryGrammian}
		A necessary condition for $T$ to be UECSM is that
		$U^*U$ and $V^*V$ have the same eigenvalues, repeated according to multiplicity.
	\end{Corollary}

	\begin{proof}
		If $T$ is UECSM, then \eqref{eq-Grammian} holds.
		Since $U^*U$ is a positive matrix, it follows that $U^*U$ and $(U^*U)^t$ are both unitarily
		equivalent to the same diagonal matrix, whence \eqref{eq-Grammian} implies that $U^*U$ and $V^*V$
		are unitarily equivalent.
	\end{proof}

	We should remark that Example \ref{ExampleCounter} in Section \ref{SectionExamples}
	reveals that passing the Grammian Test is insufficient for a matrix to be UECSM.
	On the other hand, we show in Section \ref{SectionConstructing} that
	\eqref{eq-Grammian} is both necessary and \emph{sufficient} for $T$ to be UECSM.  
	
	Taking the determinant of both sides of 
	\eqref{eq-Grammian} immediately yields the following:

	\begin{Corollary}[Parallelepiped Test]\label{CorollaryPT}
		Maintaining the notation above, if $|\det U| \neq |\det V|$,
		then $T$ is not UECSM.
	\end{Corollary}

	The name \emph{Parallelepiped Test} stems from the fact that
	$|\det U|^{\frac{1}{2}}$ and $|\det V|^{\frac{1}{2}}$ can be interpreted as the volumes of the generalized
	parallelepipeds in $\C^n$ spanned by the vectors $u_1,u_2,\ldots,u_n$
	and $v_1,v_2,\ldots,v_n$, respectively.
	
	The following example illustrates the preceding ideas:	
	
	\begin{Example}
		We claim that the matrix
		\begin{equation*}
			T = \megamatrix{0}{1}{1}{0}{1}{0}{0}{0}{2}
		\end{equation*}
		is not UECSM.  Letting $\lambda_1 = 0$, $\lambda_2 = 1$, and $\lambda_3 =2$ we obtain the corresponding
		normalized eigenvectors
		\begin{equation}\label{eq-PTu}
			u_1 = \threevector{1}{0}{0}, \qquad
			u_2 = \threevector{ \frac{1}{\sqrt{2}} }{ \frac{1}{\sqrt{2}}}{0},\qquad
			u_3 = \threevector{ \frac{1}{\sqrt{5}} }{0}{ \frac{2}{\sqrt{5}}},
		\end{equation}
		and
		\begin{equation}\label{eq-PTv}
			v_1 = \threevector{-\frac{2}{3}}{ \frac{2}{3} }{\frac{1}{3}}, \qquad
			v_2 = \threevector{ 0}{1}{0}, \qquad
			v_3 = \threevector{0}{0}{1},
		\end{equation}
		of $T$ and $T^*$, respectively.  Setting
		\begin{equation*}
			U = 
			\left(
			\begin{array}{c|c|c}
				 1 & \frac{1}{\sqrt{2}} & \frac{1}{\sqrt{5}} \\
				 0 & \frac{1}{\sqrt{2}} & 0 \\
				 0 & 0 & \frac{2}{\sqrt{5}}
			\end{array}
			\right),
			\qquad
			V=
			\left(
			\begin{array}{c|c|c}
				 -\frac{2}{3} & 0 & 0 \\
				 \frac{2}{3} & 1 & 0 \\
				 \frac{1}{3} & 0 & 1
			\end{array}
			\right),
		\end{equation*}
		we immediately find that
		\begin{equation*}
		|\det U| = \sqrt{ \frac{2}{5}} \neq \frac{2}{3} = | \det V|
		\end{equation*}
		whence it follows from the Parallelepiped Test that $T$ is not UECSM.

		Moreover, we also have
		\begin{equation}\label{eq-UUVV}
		(U^*U)^t=
		\begin{pmatrix}
		 1 & \frac{1}{\sqrt{2}} & \frac{1}{\sqrt{5}} \\
		 \frac{1}{\sqrt{2}} & 1 & \frac{1}{\sqrt{10}} \\
		 \frac{1}{\sqrt{5}} & \frac{1}{\sqrt{10}} & 1
		\end{pmatrix},
		\qquad
		V^*V=
		\begin{pmatrix}
		 1 & \frac{2}{3} & \frac{1}{3} \\
		 \frac{2}{3} & 1 & 0 \\
		 \frac{1}{3} & 0 & 1
		\end{pmatrix}
		\end{equation}
		whence, by considering the moduli of the off-diagonal entries in \eqref{eq-UUVV},
		it is clear that no diagonal unitary matrix $A$ exists which satisfies \eqref{eq-Grammian}.
		Thus the Grammian Test once again establishes that $T$ is not UECSM.

		Finally, let us take this opportunity to illustrate the Angle Test, which is less computationally intensive
		than either the Parallelepiped Test or the Grammian Test.
		A short calculation based upon the data
		\eqref{eq-PTu} and \eqref{eq-PTv} reveals that
		\begin{equation*}
		|\inner{ u_1,u_2}| = \frac{1}{\sqrt{2}} \neq \frac{2}{3} = | \inner{ v_1,v_2}|
		\end{equation*}
		whence $T$ is not UECSM.
		This can be also seen directly by examining the $(1,2)$ entry of the matrices in \eqref{eq-UUVV}.
	\end{Example}	
	
	It is important to remark that none of the conditions described in Theorem \ref{TheoremWeak},
	Corollary \ref{CorollaryGrammian}, or Corollary \ref{CorollaryPT}, are sufficient for $T$
	to be UECSM.  This is illustrated in a series of rather involved computations
	(see Example \ref{ExampleCounter}) that we postpone until later.  In Section \ref{SectionStrong}
	we remedy this situation and provide a test which is both necessary and sufficient.

\section{Constructing a Conjugation}\label{SectionConstructing}

	Under our running hypotheses, it turns out that the condition \eqref{eq-AlphaDefinition} is sufficient for 
	$T$ to be UECSM (in particular, so is the Gram matrix condition \eqref{eq-Grammian}).
	The following lemma is the main workhorse upon which the rest of this note is based:
	
	\begin{Lemma}\label{LemmaConstruct}
		Let
		\smallskip
		\begin{enumerate}\addtolength{\itemsep}{0.5\baselineskip}
			\item $T$ be a $n \times n$ matrix with distinct eigenvalues $\lambda_1, \lambda_2, \ldots, \lambda_n$,
			\item $u_1,u_2,\ldots, u_n$ denote normalized eigenvectors of $T$ 
				corresponding to the eigenvalues $\lambda_i$,
			\item $v_1,v_2,\ldots, v_n$ denote normalized eigenvectors of $T^*$ 
				corresponding to the eigenvalues $\overline{\lambda_i}$.		
			\end{enumerate}
		\smallskip
		If unimodular constants $\alpha_1,\alpha_2,\ldots, \alpha_n$ exist such that 
		\begin{equation}\label{eq-AlphaCondition}
			\inner{u_i,u_j} = \overline{\alpha_i} \alpha_j \inner{v_j,v_i}
		\end{equation}
		holds for $1 \leq i < j \leq n$, then $T$ is UECSM.
	\end{Lemma}
	
	\begin{proof}
		First observe that if \eqref{eq-AlphaCondition} holds for $1 \leq i < j \leq n$, then \eqref{eq-AlphaCondition}
		holds whenever $1 \leq i, j \leq n$ by symmetry and the fact that we are considering normalized eigenvectors.	
		Let $Cu_i = \alpha_i v_i$ for $i = 1,2,\ldots, n$ and extend this by conjugate-linearity to all of $\C^n$.
		We intend to show that $C$ is a conjugation with respect to which $T$ is $C$-symmetric.
		Since $C$ is conjugate-linear by definition, 
		it suffices to show that $C$ is involutive and isometric.
		\medskip

		\noindent\textbf{Step 1}:  Show that $C$ is involutive (i.e., $C^2 = I$).
		\medskip

		Since $C^2$ is linear, it suffices to verify that $C^2 u_i = u_i$ for $i = 1,2,\ldots,n$.
		By \eqref{eq-ReconstructU} it follows that
		\begin{equation}\label{eq-FirstStep}
			Cu_i
			= \alpha_i v_i 
			= \alpha_i \sum_{j=1}^n \frac{ \inner{v_i,v_j} }{ \inner{u_j, v_j} } u_j
		\end{equation} 
		whence 
		\begin{align*}
			C^2 u_i
			&= C \left(\alpha_i \sum_{j=1}^n \frac{ \inner{v_i,v_j} }{ \inner{u_j, v_j} } u_j \right) 
				&&\text{\small by \eqref{eq-FirstStep}}\\ 
			&= \overline{\alpha_i} \sum_{j=1}^n \frac{ \inner{v_j,v_i} }{ \inner{v_j, u_j} } Cu_j 
				&&\text{\small conjugate-linearity}\\
			&= \overline{\alpha_i} \sum_{j=1}^n \frac{ \inner{v_j,v_i} }{ \inner{v_j, u_j} } \alpha_j v_j 
				&&\text{\small definition of $C$}\\
			&= \overline{\alpha_i} \sum_{j=1}^n \frac{ \alpha_i \overline{ \alpha_j} \inner{u_i,u_j} }{ \inner{v_j, u_j} } \alpha_j v_j 
				&&\text{\small by \eqref{eq-AlphaCondition}}\\
			&= \sum_{j=1}^n \frac{ \inner{u_i,u_j} }{ \inner{v_j, u_j} } v_j   \\
			&= u_i.  
				&& \text{\small by \eqref{eq-ReconstructV}}
		\end{align*}
		Thus $C$ is involutive.
		\medskip

		\noindent\textbf{Step 2}:  Show that $C$ is isometric (i.e., $\norm{Cx} = \norm{x}$ for all $x \in \C^n$).
		\medskip

			If $x = \sum_{i=1}^n c_i u_i$, then observe that 
		\begin{align*}
			\norm{x}^2
			&= \inner{x,x} \\
%			&= \inner{ \sum_{i=1}^n c_i u_i, \sum_{j=1}^n c_j u_j } \\
%			&= \sum_{i=1}^n c_i \inner{u_i, \sum_{j=1}^n c_j u_j } \\
%			&= \sum_{i=1}^n c_i \sum_{j=1}^n \overline{c_j} \inner{u_i,  u_j } \\
			&= \sum_{i,j=1}^n c_i \overline{c_j} \inner{u_i,  u_j } \\
			&= \sum_{i,j=1}^n c_i \overline{c_j} \overline{ \alpha_i} \alpha_j \inner{v_j,  v_i } 
				&&\text{by \eqref{eq-AlphaCondition}}\\
			&= \sum_{i,j=1}^n c_i \overline{c_j}  \inner{\alpha_j v_j,  \alpha_i v_i } \\
			&= \sum_{i,j=1}^n c_i \overline{c_j}  \inner{Cu_j,  Cu_i} 
				&& \text{definition of $C$}\\
			&= \sum_{i,j=1}^n    \inner{\overline{c_j} Cu_j,  \overline{c_i} Cu_i} \\
			&=  \inner{\sum_{j=1}^n   \overline{c_j} Cu_j,  \sum_{i=1}^n   \overline{c_i} Cu_i} \\
			&= \inner{Cx,Cx} \\
			&= \norm{Cx}^2.
		\end{align*}
		Thus $C$ is isometric whence $C$ is a conjugation on $\C^n$.
		\medskip

		\noindent\textbf{Step 3}:  Show that $T$ is $C$-symmetric (i.e., $T = CT^*C$).
		\medskip

		Since both $T$ and $CT^*C$ are linear, it suffices to prove that they agree on the basis
		$u_1,u_2,\ldots,u_n$.  Having shown that $C^2 = I$, it now follows from the 
		equation $Cu_i = \alpha_i v_i$ and the conjugate-linearity of $C$ that $Cv_i = \alpha_i u_i$.
		Thus
		\begin{align*}
			CT^*C u_i
			&= CT^*(\alpha_i v_i) 
			= \overline{\alpha_i} CT^*v_i 
			= \overline{\alpha_i} C \overline{\lambda_i} v_i \\ 
			&= \overline{\alpha_i} \lambda_i C v_i 
			= \overline{\alpha_i} \lambda_i \alpha_i u_i 
			= \lambda_i u_i 
			= Tu_i
		\end{align*}
		whence $T$ is $C$-symmetric and hence UECSM.
	\end{proof}

	The conjugation $C$ constructed by Lemma \ref{LemmaConstruct} can be
	concretely realized as $C = SJ$ where $S$ is a complex symmetric unitary matrix
	and $J$ denotes the canonical conjugation \eqref{eq-Canonical}
	on $\C^n$.  Let us briefly describe the construction of the matrix $S$.

	First observe that $C$ satisfies $Cu_i = \alpha_i v_i$, which is 
	easily seen to be equivalent to $Cv_i = \alpha_i u_i$ for $i=1,2,\ldots,n$.
	As before, let $U = (u_1 | u_2 | \cdots | u_n)$ and $V = (v_1 | v_2 | \cdots | v_n)$ denote the
	matrices having the vectors $u_1,u_2,\ldots, u_n$ and $v_1,v_2, \ldots, v_n$, respectively, as columns.  Since the
	columns of $U$ and $V$ form bases of $\C^n$, it follows that both of these matrices are invertible.  Next we note that
	\begin{equation*}
		V^* U 
		= 
		\begin{pmatrix}
			\overline{v_1} \\ \hline \overline{v_2} \\ \hline \vdots \\ \hline \overline{v_n} 
		\end{pmatrix}
		(u_1 | u_2 | \cdots | u_n) = 
		\begin{pmatrix}
			\inner{u_1,v_1} & & &\\
			& \inner{u_2,v_2} & & \\
			& & \ddots &\\
			&&&\inner{u_n,v_n}
		\end{pmatrix}
		= E
	\end{equation*}
	by Lemma \ref{LemmaBiorthogonality}.  Let
	\begin{equation}\label{eq-DefinitionAMatrix}
		A = 
		\begin{pmatrix}
		\alpha_1 & & &\\
		& \alpha_2 & & \\
		& & \ddots &\\
		&&&\alpha_n
		\end{pmatrix}
	\end{equation}
	and
	\begin{equation}\label{eq-EDefinition}
		D = AE^{-1} = 
		\begin{pmatrix}
		\frac{\alpha_1}{ \inner{u_1,v_1} } &&&\\
		&\frac{\alpha_2}{ \inner{u_2,v_2} } &&\\
		& & \ddots &\\
		&&&\frac{\alpha_n}{ \inner{u_n,v_n} }
		\end{pmatrix}.
	\end{equation}
	
	We claim that $C = SJ$ where
	\begin{equation}\label{eq-SClaim}
		S = UDU^t.
	\end{equation}
	To prove \eqref{eq-SClaim}, it suffices to show that the conjugate-linear operators
	 $C$ and $SJ$ agree on each of the vectors $v_i$.  
	In other words, we must show that $SJv_i = \alpha_i u_i$ for $i = 1,2,\ldots,n$.
	Letting $s_1,s_2,\ldots,s_n$ denote the standard basis for $\C^n$ we have
	\begin{align*}\qquad
		SJv_i 
		&= UDU^tJ v_i &&\text{by \eqref{eq-SClaim}}\\
		&= UAE^{-1}U^t J v_i  &&\text{by \eqref{eq-EDefinition}}\\
		&= UAE^{-1}J U^* v_i && \text{since $JU^* = U^t J$}\\
		&= UAJ \overline{E}^{-1} U^* v_i &&  \text{since $J\overline{E}^{-1} = E^{-1}J$}\\
		&= UAJ V^{-1} v_i && \text{since $V^{-1} =  \overline{E}^{-1} U^*$}\\
		&= UAJ s_i  && \text{def.~of $V$}\\
		&= UA s_i  && \text{since $Js_i = s_i$}\\
		&= U\alpha_i s_i  && \text{by \eqref{eq-DefinitionAMatrix}} \\
		&= \alpha_i u_i.  && \text{def.~of $U$}
	\end{align*}
	Thus $C = SJ$.  Since the matrix $D = AE^{-1}$ is diagonal, it is clear from \eqref{eq-SClaim} that $S$
	is symmetric.  Since $S = CJ$ is the product of two conjugations,
	it is an invertible isometry and hence unitary (see also \cite[Lem.~1]{CSO2}).

	It is worth remarking that the condition $T = CT^*C$ implies that
	$T = SJT^*SJ = ST^tJSJ = ST^tS^*$ since $S$ is symmetric (i.e., $S$ is $J$-symmetric).
	Therefore the matrix $S$ yields a unitary equivalence between $T$ and its
	transpose $T^t$.

\section{The Strong Angle Test}\label{SectionStrong}
	
	The main theorem of this article is the following necessary and sufficient condition
	for a matrix with distinct eigenvalues to be UECSM.  The procedure introduced in the following theorem
	was implemented in \texttt{Mathematica} by the first author.  We refer to this procedure as \Test. 

	\begin{Theorem}[Strong Angle Test]\label{TheoremStrong}
		If
		\smallskip
		\begin{enumerate}\addtolength{\itemsep}{0.5\baselineskip}
			\item $T$ is a $n \times n$ matrix with distinct eigenvalues $\lambda_1, \lambda_2, \ldots, \lambda_n$,
			\item $u_1,u_2,\ldots, u_n$ denote normalized eigenvectors of $T$ 
				corresponding to the eigenvalues $\lambda_1, \lambda_2, \ldots, \lambda_n$,
			\item $v_1,v_2,\ldots, v_n$ denote normalized eigenvectors of $T^*$ 
				corresponding to the eigenvalues $\overline{\lambda_1}, \overline{\lambda_2}, \ldots, \overline{\lambda_n}$,
		\end{enumerate}
		\smallskip
		then $T$ is UECSM if and only if the condition
		\begin{equation}\label{eq-Cocycle}
			\inner{u_i,u_j} \inner{u_j,u_k} \inner{u_k,u_i} 
			= \overline{ \inner{v_i,v_j}  \inner{v_j,v_k} \inner{ v_k,v_i} }
		\end{equation}
		holds whenever $1 \leq i \leq j \leq k \leq n$ and not all of $i,j,k$ are equal.\footnote{Observe that setting $k=j$ in
		condition \eqref{eq-Cocycle} leads to $|\inner{u_i,u_j}| = |\inner{v_i,v_j}|$ for $1 \leq i \leq j \leq n$.
		Thus Theorem \ref{TheoremStrong} can be viewed as an extension of the original
		Angle Test (Theorem \ref{TheoremWeak}).  Also note that if $i = j = k$, then
		\eqref{eq-Cocycle} merely asserts that $\norm{u_i} = \norm{v_i}$ which is already known from conditions (ii) and (iii).}
	\end{Theorem}
	
	\begin{proof}
		The necessity of the condition \eqref{eq-Cocycle} follows immediately from \eqref{eq-AlphaDefinition}.
		The proof that \eqref{eq-Cocycle} is sufficient for $T$ to be UECSM is more complicated.	
		First observe that if \eqref{eq-Cocycle} holds for $1 \leq i \leq j \leq k \leq n$, then \eqref{eq-Cocycle}
		holds whenever $1 \leq i,j,k \leq n$ by symmetry.
		Let us assume for the moment
		that $\inner{u_i,u_j} \neq 0$ (whence $\inner{v_j,v_i} \neq 0$) for $1 \leq i,j \leq n$.  
		Later we will relax this restriction, but for the sake of clarity it will be easier to consider this special case first.
		Under this additional hypothesis, there exist $n^2$ unimodular constants
		$\beta_{ij}$ uniquely determined by 
		\begin{equation}\label{eq-NewBeta}
			\beta_{ij} = \frac{ \inner{u_i,u_j} }{ \inner{ v_j, v_i} }
		\end{equation}
		for $1 \leq i,j \leq n$.  Since $\norm{u_i} = \norm{v_i} = 1$ by hypotheses (ii) and (iii), it follows immediately that
		that $\beta_{ii}= 1$ for $1 \leq i \leq n$.  Moreover, we also have
		\begin{equation*}
			\beta_{ij} \inner{v_j,v_i} 
			= \inner{u_i,u_j} 
			= \overline{ \inner{ u_j, u_i} } 
			= \overline{ \beta_{ji} \inner{v_i,v_j} } 
			= \overline{ \beta_{ji} } \inner{v_j, v_i},
		\end{equation*}
		whence $\beta_{ij} = \overline{ \beta_{ji} }$.  In other words, the matrix $B = ( \beta_{ij} )_{i,j=1}^n$
		is selfadjoint and has constant diagonal $1$.  Suppose for the moment that, 
		based on the hypothesis \eqref{eq-Cocycle},
		we are able to establish that $B$ enjoys a factorization of the form
		\begin{equation}\label{eq-Factor}
			\begin{pmatrix}
				1	&	\beta_{12}	&	\beta_{13}	& \cdots & \beta_{1n} \\
				\beta_{21}	&	1	&	\beta_{23}	& \cdots & \beta_{2n} \\
				\beta_{31}	&	\beta_{32}	&	1	& \cdots & \beta_{3n} \\
				\vdots & \vdots & \vdots & \ddots & \vdots \\
				\beta_{n1}	&	\beta_{n2}	&	\beta_{n3} & \cdots & 1
			\end{pmatrix}	
			=
			\begin{pmatrix}
				\overline{ \alpha_1} \\ \overline{ \alpha_2} \\ \vdots \\ \overline{ \alpha_n}
			\end{pmatrix}
			\begin{pmatrix}	
				\alpha_1 & \alpha_2 & \cdots & \alpha_n
			\end{pmatrix}
		\end{equation}
		(i.e., suppose that we are able to show that $B$ is positive and has rank one).  By \eqref{eq-NewBeta}
		and the preceding factorization \eqref{eq-Factor} it would then follow that the unimodular constants 
		$\alpha_1, \alpha_2, \ldots, \alpha_n$ satisfy
		\begin{equation}\label{eq-Split}
			\inner{u_i,u_j} = \overline{\alpha_i} \alpha_j \inner{v_j,v_i}
		\end{equation}
		for $1 \leq i,j \leq n$.  At this point, we could invoke Lemma \ref{LemmaConstruct}
		to conclude that $T$ is UECSM.
		\medskip

		The difficulty in the approach outlined above lies in the fact that some of the inner products
		$\inner{u_i,u_j}$ or $\inner{v_j,v_i}$ may vanish.  If this occurs, then we cannot immediately
		consider the associated unimodular constants $\beta_{ij}$ defined by \eqref{eq-NewBeta}
		since applying \eqref{eq-Cocycle} with $k = i$ implies that 
		$\inner{ v_j,v_i} = 0$ if and only if $\inner{u_i,u_j} = 0$.  	
		On the other hand, observe that the hypothesis \eqref{eq-Cocycle} implies that
		\begin{equation*}
			\beta_{ij} \beta_{jk} \beta_{ki}
			= \frac{ \inner{u_i,u_j} }{\inner{v_j,v_i}}  \frac{ \inner{u_j,u_k} }{ \inner{v_k,v_j} }
			\frac{ \inner{u_k,u_i} }{   \inner{ v_i,v_k}  } = 1
		\end{equation*}
		holds whenever $\beta_{ij}, \beta_{jk}, \beta_{ki}$ are well-defined by \eqref{eq-NewBeta}.
		In light of the fact that each $\beta_{ij}$ is unimodular, we obtain the following 
		\emph{multiplicative property}
		\begin{equation}\label{eq-Composition}
			\beta_{ij} = \beta_{ik} \beta_{kj}
		\end{equation}
		whenever the expressions above are well-defined by \eqref{eq-NewBeta}.

		Regarding the matrix $B = ( \beta_{ij} )_{i,j=1}^n$ as being only partially defined
		by \eqref{eq-NewBeta}, suppose for the moment that we are able to define unimodular 
		constants $\beta_{ij}$ for those $i$ and $j$ for which 
		$\inner{u_i,u_j} = \inner{v_j,v_i} = 0$ such that the multiplicative property
		\eqref{eq-Composition} holds for all $1 \leq i,j,k \leq n$.  Under this hypothesis, we claim that
		the matrix $B = (\beta_{ij} )_{i,j=1}^n$ has a factorization of the form \eqref{eq-Factor}.
		Indeed, use \eqref{eq-Composition} and the fact that $\beta_{ji} = \overline{ \beta_{ij} }$ for $1 \leq i,j \leq n$
		to conclude that
		\begin{equation}\label{eq-FactorTwo}
			B=
			\begin{pmatrix}
				1	&	\beta_{12}	&	\beta_{13}	& \cdots & \beta_{1n} \\
				\beta_{21}	&	1	&	\beta_{23}	& \cdots & \beta_{2n} \\
				\beta_{31}	&	\beta_{32}	&	1	& \cdots & \beta_{3n} \\
				\vdots & \vdots & \vdots & \ddots & \vdots \\
				\beta_{n1}	&	\beta_{n2}	&	\beta_{n3} & \cdots & 1
			\end{pmatrix}
			=
			\begin{pmatrix}	
				1	\\	\overline{\beta_{12} }	\\	\overline{ \beta_{13}} 
				\\ \vdots \\ \overline{ \beta_{1n} }\\
			\end{pmatrix}
			\begin{pmatrix}	
				1	&	\beta_{12}	&	\beta_{13}	& \cdots & \beta_{1n} \\
			\end{pmatrix}
			.
		\end{equation}
		As suggested by \eqref{eq-Factor}, we now define
		the unimodular constants $\alpha_1, \alpha_2, \ldots, \alpha_n$ by setting
		$\alpha_i = \beta_{1i}$ for $1 \leq i \leq n$.  		
		Next observe that
		\begin{equation*}
			\overline{\alpha_i} \alpha_j
			= \overline{\beta_{1i} } \beta_{1j}
			= \beta_{i1} \beta_{1j}			
			=\beta_{ij} = \frac{ \inner{u_i,u_j} }{ \inner{ v_j,v_i} }
		\end{equation*}
		holds whenever $\beta_{ij}$ is well-defined by \eqref{eq-NewBeta}.
		Thus the desired condition \eqref{eq-Split} holds for all $1 \leq i,j \leq n$ (since it
		holds trivially if $\inner{u_i,u_j} = \inner{v_j,v_i} = 0$) and $T$ is UECSM
		by Lemma \ref{LemmaConstruct}.
		\medskip
		
		To complete the proof of Theorem \ref{TheoremStrong}
		it suffices to demonstrate a procedure by which we may
		define unimodular constants $\beta_{ij}$ for those $i$ and $j$ for which 
		$\inner{u_i,u_j} = \inner{v_j,v_i} = 0$ such that the multiplicative property
		\eqref{eq-Composition} holds for all $1 \leq i,j \leq n$.  This will lead us to the
		desired matrix factorization \eqref{eq-FactorTwo}.
		
			 To define the constants $\beta_{ij}$ we employ an inductive procedure. 
		Consider the partially defined $n \times n$ matrix
		\begin{equation}\label{eq-Goal}
			\small
			\left(
				\begin{array}{ccccc|ccc}
					1 & \beta_{12} & \beta_{13} & \cdots & \beta_{1r} & * & \cdots & *\\
					\beta_{21} & 1 & \beta_{23} & \cdots & \beta_{2r} & * & \cdots & *\\
					\beta_{31} & \beta_{32} & 1 & \cdots & \beta_{3r} & * & \cdots & *\\
					\vdots & \vdots & \vdots & \ddots & \vdots &\vdots & \ddots & \vdots\\
					\beta_{r1} & \beta_{r2} & \beta_{r3} & \cdots & 1 & * & \cdots & *\\
					\hline
					* & * & * & \cdots & * & 1 & \cdots & *\\
					\vdots & \vdots & \vdots & \ddots & \vdots &\vdots & \ddots & \vdots\\
					* & * & * & \cdots & * & * & \cdots & 1\\
				\end{array}
			\right)
		\end{equation}
		where $*$ indicates either an entry $\beta_{ij}$ already defined by \eqref{eq-NewBeta} 
		or an entry that is not defined in terms of \eqref{eq-NewBeta} because $\inner{u_i,u_j} = \inner{ v_j,v_i} = 0$.
		As our inductive hypothesis, we assume that the multiplicative property \eqref{eq-Composition} is satisfied by
		all triples $\beta_{ij}, \beta_{ik}, \beta_{kj}$ for which $1 \leq i,j,k \leq r$.

		To complete the proof Theorem \ref{TheoremStrong},
		we must devise a way to fill out the undefined entries in \eqref{eq-Goal} with unimodular constants
		$\beta_{ij}$ in such a way that \eqref{eq-Composition} holds for these new entries.  
		There are two cases to consider:
		\medskip

		\noindent\textbf{Case 1}: 	Suppose that there exists an entry 
		$\beta_{i(r+1)}$ with $1 \leq i \leq r$ in \eqref{eq-Goal}
		that is already defined by \eqref{eq-NewBeta}.
		Without loss of generality, we may assume that it is the 
		$\beta_{1(r+1)}$ is the entry that is well-defined by \eqref{eq-NewBeta}
		since this situation may be obtained by permuting the indices $1,2,\ldots, r$ and relabeling the eigenvectors
		$u_1,u_2,\ldots,u_{r}; v_1,v_2,\ldots,v_{r}$.  We
		are thus left with the partially completed matrix
		\begin{equation*}
			\small
			\left(
				\begin{array}{ccccc|c|ccc}
					1 & \beta_{12} & \beta_{13} & \cdots & \beta_{1r} & \beta_{1(r+1)} &* & \cdots & *\\
					\beta_{21} & 1 & \beta_{23} & \cdots & \beta_{2r} & * & * & \cdots & *\\
					\beta_{31} & \beta_{32} & 1 & \cdots & \beta_{3r} & * & * & \cdots & *\\
					\vdots & \vdots & \vdots & \ddots & \vdots &\vdots & \vdots & \ddots & \vdots\\
					\beta_{r1} & \beta_{r2} & \beta_{r3} & \cdots & 1 &  * &* & \cdots & *\\
					\hline
					\beta_{(r+1) 1} & * & * & \cdots 
						& * &1 &* & \cdots & *\\
					\hline
					* & * & * & \cdots & * & *&1 & \cdots & *\\
					\vdots & \vdots & \vdots & \ddots & \vdots &\vdots & \vdots & \ddots & \vdots\\
					* & * & * & \cdots & * & * & *& \cdots & 1\\
				\end{array}
			\right).
		\end{equation*}
		For each entry $\beta_{i(r+1)}$ with $2 \leq i \leq r$ (i.e., the entries immediately below
		$\beta_{1(r+1)}$ and above the $1$ on the main diagonal)
		there are two possibilities:

		\medskip
		\noindent\textbf{Subcase 1.1}:
		If $\beta_{i(r+1)}$ is already well-defined by \eqref{eq-NewBeta}, then do nothing.

		\medskip
		\noindent\textbf{Subcase 1.2}:
		If $\beta_{i(r+1)}$ cannot be defined by \eqref{eq-NewBeta} because 
		$\inner{u_i,u_j} = \inner{ v_j,v_i} = 0$, then let
		\begin{equation}\label{eq-TauStep}
			\beta_{i(r+1)} := \beta_{i1} \beta_{1(r+1)}
		\end{equation}
		to obtain the partially defined matrix\footnote{The entries 
		$\beta_{(r+1) 1} , \beta_{(r+1) 2} , \beta_{(r+1) 3} , \ldots,   \beta_{r(r+1)}$ in the $(r+1)$st
		row are defined by conjugate symmetry:  $\beta_{ij} = \overline{\beta_{ji}}$.}
		\begin{equation}\label{eq-OneStep}
			\small
			\left(
				\begin{array}{ccccc|c|ccc}
					1 & \beta_{12} & \beta_{13} & \cdots & \beta_{1r} & \beta_{1(r+1)} &* & \cdots & *\\
					\beta_{21} & 1 & \beta_{23} & \cdots & \beta_{2r} & \beta_{2(r+1)} & * & \cdots & *\\
					\beta_{31} & \beta_{32} & 1 & \cdots & \beta_{3r} & \beta_{3(r+1)} & * & \cdots & *\\
					\vdots & \vdots & \vdots & \ddots & \vdots &\vdots & \vdots & \ddots & \vdots\\
					\beta_{r 1} & \beta_{r 2} & \beta_{r3} & \cdots & 1 &  \beta_{r(r+1)} &* & \cdots & *\\
					\hline
					\beta_{(r+1) 1} & \beta_{(r+1) 2} & \beta_{(r+1) 3} & \cdots 
						&   \beta_{(r+1)r} &1 &* & \cdots & *\\
					\hline
					* & * & * & \cdots & * & *&1 & \cdots & *\\
					\vdots & \vdots & \vdots & \ddots & \vdots &\vdots & \vdots & \ddots & \vdots\\
					* & * & * & \cdots & * & * & *& \cdots & 1\\
				\end{array}
			\right).
		\end{equation}
		\medskip

		\noindent\textbf{Case 2}: 	Suppose that there does not exist an entry 
		$\beta_{i(r+1)}$ with $1 \leq i \leq r$ in \eqref{eq-Goal}
		that is already defined by \eqref{eq-NewBeta}.
		In other words, suppose that $\inner{u_i,u_{r+1}} = \inner{v_{r+1},v_i} = 0$
		whenever $1 \leq i \leq r$.  We
		are thus left with the partially completed matrix
		\begin{equation*}
			\small
			\left(
				\begin{array}{ccccc|c|ccc}
					1 & \beta_{12} & \beta_{13} & \cdots & \beta_{1r} & * &* & \cdots & *\\
					\beta_{21} & 1 & \beta_{23} & \cdots & \beta_{2r} & * & * & \cdots & *\\
					\beta_{31} & \beta_{32} & 1 & \cdots & \beta_{3r} & * & * & \cdots & *\\
					\vdots & \vdots & \vdots & \ddots & \vdots &\vdots & \vdots & \ddots & \vdots\\
					\beta_{r 1} & \beta_{r 2} & \beta_{r 3} & \cdots & 1 &  * &* & \cdots & *\\
					\hline
					* & * & * & \cdots 
						& * &1 &* & \cdots & *\\
					\hline
					* & * & * & \cdots & * & *&1 & \cdots & *\\
					\vdots & \vdots & \vdots & \ddots & \vdots &\vdots & \vdots & \ddots & \vdots\\
					* & * & * & \cdots & * & * & *& \cdots & 1\\
				\end{array}
			\right).
		\end{equation*}
		Select a complex number of unit modulus and assign this value to $\beta_{1(r+1)}$.
		Having done this, we define $\beta_{i(r+1)}$ for $1\leq i \leq r$ as in \eqref{eq-TauStep}
		to obtain a partially defined matrix
		of the form \eqref{eq-OneStep}.
		\medskip
		
		To wrap-up the proof, we must show that in either case \eqref{eq-TauStep} 
		defines the new entries $\beta_{i(r+1)}$ in a manner which is consistent with the multiplicative
		property \eqref{eq-Composition}.
		For $1 \leq i,k \leq r$ we employ the definition \eqref{eq-TauStep} to find that
		\begin{align*}
			\beta_{ik} \beta_{k(r+1)} 
			&= (\beta_{i1} \beta_{1k})( \beta_{k1} \beta_{1(r+1)})   && \text{by inductive hypothesis and \eqref{eq-TauStep}}\\
			&= \beta_{i1} (\beta_{1k} \overline{\beta_{1k}}) \beta_{1(r+1)} &&\text{hermitian symmetry}\\
			&= \beta_{i1} \beta_{1(r+1)} && \text{since $|\beta_{1k}|=1$}\\
			&= \beta_{i(r+1)}.  &&\text{by \eqref{eq-TauStep}}
		\end{align*}
		Thus \eqref{eq-TauStep} defines $\beta_{i(r+1)}$ for $1 \leq i,j \leq r+1$
		in a manner consistent with \eqref{eq-Composition}.
	
		Starting with the upper left $1 \times 1$ block, repeated applications of the preceding inductive procedure
		eventually yields an $n \times n$ matrix $B = ( \beta_{ij} )_{i,j=1}^n$ whose entries
		satisfy the required multiplicative condition \eqref{eq-Composition}.  This concludes
		the proof of Theorem \ref{TheoremStrong}.
	\end{proof}

	\begin{Corollary}\label{Corollary2x2}
		Every $2 \times 2$ matrix is UECSM.
	\end{Corollary}

	\begin{proof}
		Let $T$ be a $2 \times 2$ matrix.
		If $T$ has a repeated eigenvalue
		$\lambda$, then by Schur's Theorem on Unitary Upper Triangularization, it follows that
		$T - \lambda I$ is unitarily equivalent to a scalar multiple of a $2 \times 2$ nilpotent Jordan matrix.
		This Jordan matrix is $C$-symmetric with respect to $C(z_1,z_2) = (\overline{z_2}, \overline{z_1})$
		whence $T$ is UECSM.  We therefore restrict our attention to the 
		case where $T$ has two distinct eigenvalues.  Upon applying Schur's Theorem, 
		subtracting a suitable multiple of the identity, and normalizing, we may assume that
		\begin{equation}\label{eq-2x2SchurReduction}
			T = \minimatrix{0}{a}{0}{1}
		\end{equation}
		for some complex constant $a$.  A short computation reveals that 
		normalized eigenvectors for $T$ corresponding to the eigenvalues $\lambda_1 = 0$ and $\lambda_2=1$ are
		\begin{equation}\label{eq-2x2u}
			u_1 = \twovector{1}{0}, \quad 
			u_2 = \twovector{\frac{a}{ \sqrt{1 + \left | a \right |^2}}}{\frac{1}{ \sqrt{1 + \left | a \right |^2}}}.
		\end{equation}
		Similarly, we find that corresponding normalized eigenvectors for $T^*$ are given by
		\begin{equation}\label{eq-2x2v}
			v_1 = \twovector{\frac{ 1}{ \sqrt{1 + | a |^2}} }{\frac{ -\overline{a} }{ \sqrt{1 +  |a |^2}} }, \qquad
			v_2 = \twovector{0}{1}.
		\end{equation}
		By Theorem \ref{TheoremStrong}, $T$ is UECSM if and only if \eqref{eq-Cocycle} holds for
		all $1 \leq i \leq j \leq k \leq 2$ such that not all of $i,j,k$ are equal.  This leaves us 
		only two ordered triples $(i,j,k)$ to consider:
		$(1,1,2)$ and $(1,2,2)$.  These values of $i,j,k$ both lead to the condition
		$| \inner{u_1,u_2} | = | \inner{ v_1,v_2} |$ which needs to be verified.  Since
		\begin{equation*}
			| \inner{u_1,u_2} | = \frac{ |a|}{ \sqrt{1 +  |a|^2}} = | \inner{ v_1,v_2} |
		\end{equation*}
		follows immediately from \eqref{eq-2x2u} and \eqref{eq-2x2v}, we conclude that $T$ is UECSM.
	\end{proof}

	The preceding corollary has been proved in a number of different ways by several different authors.
	For instance, one can reduce to the special case \eqref{eq-2x2SchurReduction} as above
	and then construct the corresponding conjugation by straightforward computation \cite[Ex.~6]{CSOA}.  
	The procedure developed by J.~Tener can also be used to establish Corollary \ref{Corollary2x2}  \cite[Cor.~3]{Tener}.
	We should also mention that Corollary \ref{Corollary2x2} is the byproduct of more sophisticated theorems.  
	For instance, it follows immediately from N.~Chevrot, E.~Fricain, and D.~Timotin's study of the characteristic 
	functions of complex symmetric contractions \cite[Cor.~3.3]{Chevrot}.  
	More recently, the second author and W.~Wogen established that
	every \emph{binormal} operator (i.e., an operator that is unitarily equivalent to a $2 \times 2$ block operator whose entries
	are commuting normal operators) is complex symmetric \cite{SNCSO}.  Corollary \ref{Corollary2x2} is
	a special case of this result.

\section{A few examples}\label{SectionExamples}

	To illustrate the preceding ideas, we devote this section to the detailed consideration of
	several examples.  In particular, Example \ref{ExampleCounter} demonstrates that 
	none of the simple conditions given in Section \ref{SectionAngle} is sufficient for $T$ to be UECSM.

	\begin{Example}\label{Example2x2}
		Let $T$ be a $2 \times 2$ matrix with eigenvalues $\lambda_1 \neq \lambda_2$.
		As before, let $u_1,u_2$ denote normalized eigenvectors of $T$ corresponding to the eigenvalues $\lambda_1, \lambda_2$,
		respectively and let $v_1,v_2$ denote normalized eigenvectors of $T^*$
		corresponding to the eigenvalules $\overline{ \lambda_1}, \overline{\lambda_2}$, respectively.

		By Corollary \ref{Corollary2x2}, we know that $T$ is UECSM and hence 
		$|\inner{u_1,u_2}| = |\inner{v_2,v_1}|$.  We may therefore find 
		unimodular constants $\alpha_1$ and $\alpha_2$ such that
		\begin{equation*}
			\inner{u_1,u_2} = \overline{ \alpha_1} \alpha_2 \inner{ v_2, v_1}.
		\end{equation*}
		For instance, if $T$ is normal, then
		we may simply set $\alpha_1 = \alpha_2 = 1$ since
		$\inner{u_1,u_2} = \inner{v_2,v_1} = 0$.  
		Letting
		\begin{equation*}
			U = \left(
			\begin{array}{c|c}
			u_{11} & u_{12} \\
			u_{21} & u_{22}
			\end{array}
			\right)
		\end{equation*}
		denote the matrix whose columns are the eigenvectors $u_1,u_2$ of $T$ we follow
		the procedure outlined at the end of Section \ref{SectionConstructing} 
		to construct the conjugation
		\begin{equation}\label{eq-2x2SMatrixMaster}
			C \twovector{x}{y} =
			\minimatrix
			{  \dfrac{ \alpha_1 u_{11}^2}{ \inner{u_1,v_1} }  + \dfrac{ \alpha_2 u_{12}^2}{ \inner{u_2,v_2} } }
			{  \dfrac{ \alpha_1 u_{11} u_{21}}{ \inner{u_1,v_1} }  + \dfrac{ \alpha_2 u_{12} u_{22}}{ \inner{u_2,v_2} } }
			{  \dfrac{ \alpha_1 u_{11} u_{21}}{ \inner{u_1,v_1} }  + \dfrac{ \alpha_2 u_{12} u_{22}}{ \inner{u_2,v_2} } }
			{  \dfrac{ \alpha_1 u_{21}^2}{ \inner{u_1,v_1} }  + \dfrac{ \alpha_2 u_{22}^2}{ \inner{u_2,v_2} } }
			\twovector{ \overline{x} }{ \overline{y} }
		\end{equation}
		with respect to which $T$ is $C$-symmetric.
	\end{Example}

	\begin{Example}
		Applying the preceding formula to the matrix \eqref{eq-2x2SchurReduction}, where $a \neq 0$,
		and using the data \eqref{eq-2x2u} and \eqref{eq-2x2v} we find that
		\begin{equation*}
			\inner{ u_1, u_2} = \frac{ \overline{a} }{ \sqrt{1 + |a|^2} }, \qquad
			\inner{v_2,v_1} =  -\frac{ a }{ \sqrt{1 + |a|^2} }.
		\end{equation*}
		Following the notation of Example \ref{Example2x2}, note that
		\begin{align*}
			u_{11} &= 1   &u_{12} &= \frac{a}{ \sqrt{1 + |a|^2} }, \\
			u_{21} &= 0   &u_{22} &= \frac{1}{ \sqrt{1 + |a|^2} }.
		\end{align*}
		One possible solution\footnote{The other solutions will simply yield a unimodular multiple of $S$.} to the equation
		\begin{equation*}
			\inner{u_1,u_2} = \overline{ \alpha_1} \alpha_2 \inner{v_2,v_1}
		\end{equation*}
		is given by
		\begin{equation*}
			\alpha_1 = \frac{a}{|a|}, \qquad
			\alpha_2 = - \frac{ \overline{a} }{|a|}.
		\end{equation*}
		Plugging this data into \eqref{eq-2x2SMatrixMaster} we obtain the conjugation
		\begin{equation*}
			C \twovector{x}{y}=
			\minimatrix{ \frac{ a/|a| }{ \sqrt{1+ |a|^2} } }{ - \frac{ |a| }{\sqrt{1+|a|^2}}}
			{ - \frac{ |a| }{\sqrt{1+|a|^2}}}{  \frac{ -\overline{a}/|a| }{ \sqrt{1 + |a|^2} } }
			\twovector{ \overline{x} }{ \overline{y}}
		\end{equation*}
		with respect to which the matrix \eqref{eq-2x2SchurReduction} is $C$-symmetric.
	\end{Example}

	\begin{Example}
		In \cite[Ex.~3]{Tener}, the matrix 
		\begin{equation*}
			T = \megamatrix{0}{7}{0}{0}{1}{-5}{0}{0}{6}
		\end{equation*}
		is demonstrated to be UECSM via Tener's \UECSMTest{}.  For the sake of comparison,
		let us also consider this matrix using the techniques discussed above.  Letting 
		$\lambda_0 = 6$, $\lambda_1 = 1$, and $\lambda_2 = 0$, we obtain the corresponding
		normalized eigenvectors
		\begin{equation*}
			u_1 = \threevector{ - \frac{7}{11} }{ - \frac{6}{11} }{ \frac{6}{11}},\qquad
			u_2 = \threevector{ \frac{7}{5\sqrt{2}} }{ \frac{1}{ 5\sqrt{2}}}{0}, \qquad
			u_3 = \threevector{1}{0}{0}
		\end{equation*}
		of $T$ and
		\begin{equation*}
			v_1 = \threevector{0}{0}{1}, \qquad
			v_2 = \threevector{0}{ \frac{1}{\sqrt{2}} }{ \frac{1}{\sqrt{2} } }, \qquad
			v_3 = \threevector{ - \frac{6}{55} }{ \frac{42}{55} }{ \frac{7}{11} }
		\end{equation*}
		of $T^*$, respectively.  A short computation reveals that
		\begin{align*}
			\inner{u_1,u_2} &= -\frac{1}{\sqrt{2}}, 
			&\inner{u_2,u_3} &= \frac{7}{5 \sqrt{2}},
			&\inner{u_3,u_1} &= -\frac{7}{11},\\
			\inner{v_1,v_2} &= \frac{1}{\sqrt{2}}, 
			&\inner{v_2,v_3} &= \frac{7}{5 \sqrt{2}},
			&\inner{v_3,v_1} &= \frac{7}{11},
		\end{align*}
		whence it is clear that \eqref{eq-Cocycle} holds for all triples
		\begin{equation*}
			(i,j,k) = (1,1,2), (1,1,3), (1,2,2), (1,2,3), (1,3,3), (2,2,3), (2,3,3)
		\end{equation*}
		required by \Test.  In particular, this proves that $T$ is UECSM.

		The corresponding matrices $U = (u_1 | u_2 | u_3)$ and $V = (v_1 | v_2| v_3)$ are
		\begin{equation*}
			U = 
			\left(
			\begin{array}{c|c|c}
				 -\frac{7}{11} & \frac{7}{5 \sqrt{2}} & 1 \\
				 -\frac{6}{11} & \frac{1}{5 \sqrt{2}} & 0 \\
				 \frac{6}{11} & 0 & 0
			\end{array}
			\right)
			\qquad
			V = 
			\left(
			\begin{array}{c|c|c}
				 0 & 0 & -\frac{6}{55} \\
				 0 & \frac{1}{\sqrt{2}} & \frac{42}{55} \\
				 1 & \frac{1}{\sqrt{2}} & \frac{7}{11}
			\end{array}
			\right).
		\end{equation*}
		As expected, $T$ passes the Parallelepiped Test (Corollary \ref{CorollaryPT}) since
		\begin{equation*}
			|\det U| = |\det V| = \frac{3\sqrt{2}}{55} .
		\end{equation*}
		Next, observe that
		\begin{equation*}
			U^*U =
			\begin{pmatrix}
				1 & -\frac{1}{\sqrt{2}} & -\frac{7}{11} \\
				-\frac{1}{\sqrt{2}} & 1 & \frac{7}{5 \sqrt{2}} \\
				-\frac{7}{11} & \frac{7}{5 \sqrt{2}} & 1
			\end{pmatrix},
			\qquad
			V^*V =
			\begin{pmatrix}
				 1 & \frac{1}{\sqrt{2}} & \frac{7}{11} \\
				 \frac{1}{\sqrt{2}} & 1 & \frac{7}{5 \sqrt{2}} \\
				 \frac{7}{11} & \frac{7}{5 \sqrt{2}} & 1
			\end{pmatrix}
		\end{equation*}	
		whence $T$ passes the Grammian Test (Corollary \ref{CorollaryGrammian}) with the 
		$A$ from \eqref{eq-Grammian} being
		\begin{equation*}
			A = \megamatrix{1}{0}{0}{0}{-1}{0}{0}{0}{-1} = \megamatrix{\alpha_1}{0}{0}{0}{\alpha_2}{0}{0}{0}{\alpha_3}.
		\end{equation*}
		In particular, this once again confirms that $T$ is UECSM.
		
		Let us, for the moment, examine the mechanics of the proof of Theorem \ref{TheoremStrong},
		which establishes the theoretical underpinnings of the procedure \Test.  Using the data above, we find that the
		matrix $B = (\beta_{ij})$ from the proof of Theorem \ref{TheoremStrong} is given by
		\begin{align*}
			B 
			= \frac{ (U^*U)^t}{V^*V}
			= \megamatrix{1}{-1}{-1}{-1}{1}{1}{-1}{1}{1} 
			= \threevector{1}{-1}{-1}(1 \,\,\,-1\,\,\,-1)
		\end{align*}
		whence we again read that $\alpha_1 = 1, \alpha_2 = -1, \alpha_3 = -1$.  We remind the reader that 
		the quotient appearing in the preceding equation is simply  the entry-by-entry quotient of 
		the matrices $(U^*U)^t$ and $V^*V$.

		Based upon the preceding calculations, we can construct the corresponding conjugation 
		$C = SJ$ where $S$ is a complex symmetric unitary matrix which is to be determined (this matrix also
		has the property that $T = ST^t S^*$).
		Following the recipe described at the end of Section \ref{SectionConstructing} we obtain
		\begin{equation*}
			E = V^*U = 
			\begin{pmatrix}
				 \frac{6}{11} & 0 & 0 \\
				 0 & \frac{1}{10} & 0 \\
				 0 & 0 & -\frac{6}{55}
			\end{pmatrix}
		\end{equation*}
		so that
		\begin{equation*}
			D = AE^{-1} = 
			\begin{pmatrix}
				 \frac{11}{6} & 0 & 0 \\
				 0 & -10 & 0 \\
				 0 & 0 & \frac{55}{6}
			\end{pmatrix}
		\end{equation*}
		Putting this all together we find that
		\begin{equation*}
			S = UDU^t = 
			\begin{pmatrix}
				 \frac{6}{55} & -\frac{42}{55} & -\frac{7}{11} \\[3pt]
				 -\frac{42}{55} & \frac{19}{55} & -\frac{6}{11} \\[3pt]
				 -\frac{7}{11} & -\frac{6}{11} & \frac{6}{11}
			\end{pmatrix}.
		\end{equation*}				
		It turns out that our $S$ differs from the corresponding matrix obtained in 
		\cite[Ex.~3]{Tener} by a unimodular multiplicative factor 
		of $\frac{ -19 + 6 i \sqrt{74} }{55}$.
	\end{Example}

	The following important example demonstrates that the
	Angle Test (Theorem \ref{TheoremWeak}), the Grammian Test (Corollary \ref{CorollaryGrammian}), and
	the Parallelepiped Test (Corollary \ref{CorollaryPT}) are insufficient to determine
	whether a given matrix is UECSM.  In particular, this demonstrates the utility of the
	Strong Angle Test (Theorem \ref{TheoremStrong}), which provides a necessary
	\emph{and} sufficient condition.

	\begin{Example}\label{ExampleCounter}
		Consider the matrix
		\begin{equation*}
			T=
			\begin{pmatrix}
				 5 & 0 & -1 & 3 \\
				 2 & 4 & 1 & 2 \\
				 2 & -2 & 6 & -2 \\
				 0 & -2 & 1 & 4
			\end{pmatrix},
		\end{equation*}
		which has the distinct eigenvalues
		\begin{equation*}
			\lambda_1 = 5+i \sqrt{5},\quad
			\lambda_2 = 5-i \sqrt{5}, \quad
			\lambda_3 = \tfrac{1}{2} (9+i \sqrt{15}), \quad
			\lambda_4 = \tfrac{1}{2} (9-i \sqrt{15}),
		\end{equation*}
		and corresponding unit eigenvectors
		\begin{equation*}
			\underbrace{
			\begin{pmatrix}
				 -\frac{2 i}{3}-\frac{1}{3 \sqrt{5}} \\
				 \frac{1}{15} \left(-5 i-2 \sqrt{5}\right) \\
				 \frac{1}{15} \left(5 i-\sqrt{5}\right) \\
				 \frac{1}{\sqrt{5}}
			\end{pmatrix}
			}_{u_1},\quad 
			\underbrace{
			\begin{pmatrix}
				 \frac{2 i}{3}-\frac{1}{3 \sqrt{5}} \\
				 \frac{1}{15} \left(5 i-2 \sqrt{5}\right) \\
				 \frac{1}{15} \left(-5 i-\sqrt{5}\right) \\
				 \frac{1}{\sqrt{5}}
			\end{pmatrix}
			}_{u_2},\quad 
			\underbrace{
			\begin{pmatrix}
				 -\frac{i \left(-5 i+\sqrt{15}\right)}{2 \sqrt{30}} \\
				 0 \\
				 \frac{i}{\sqrt{2}}+\frac{1}{\sqrt{30}} \\
				 \sqrt{\frac{2}{15}}
			\end{pmatrix}
			}_{u_3},\quad 
			\underbrace{
			\begin{pmatrix}
				 \frac{i \left(5 i+\sqrt{15}\right)}{2 \sqrt{30}} \\
				 0 \\
				 -\frac{i}{\sqrt{2}}+\frac{1}{\sqrt{30}} \\
				 \sqrt{\frac{2}{15}}
			\end{pmatrix}
			}_{u_4}.
		\end{equation*}
		The matrix $T^*$ has the eigenvalues
		\begin{equation*}
			\overline{\lambda_1} = 5-i \sqrt{5},\quad
			\overline{\lambda_2} = 5+i \sqrt{5}, \quad
			\overline{\lambda_3} = \tfrac{1}{2} (9-i \sqrt{15}), \quad
			\overline{\lambda_4} = \tfrac{1}{2} (9+i \sqrt{15}),
		\end{equation*}
		and corresponding unit eigenvectors
		\begin{equation*}
			\underbrace{
			\begin{pmatrix}
				 \frac{2}{\sqrt{15}} \\
				 \frac{-1-i \sqrt{5}}{\sqrt{15}} \\
				 \frac{1}{\sqrt{15}} \\
				 \frac{2}{\sqrt{15}}
			\end{pmatrix}
			}_{v_1},\quad
			\underbrace{
			\begin{pmatrix}
				 \frac{2}{\sqrt{15}} \\
				 \frac{i \left(i+\sqrt{5}\right)}{\sqrt{15}} \\
				 \frac{1}{\sqrt{15}} \\
				 \frac{2}{\sqrt{15}}
			\end{pmatrix}
			}_{v_2},\quad
			\underbrace{
			\begin{pmatrix}
				 \frac{i}{2}-\frac{1}{2 \sqrt{15}} \\
				 \frac{1}{10} \left(-5 i+\sqrt{15}\right) \\
				 \frac{1}{30} \left(15 i+\sqrt{15}\right) \\
				 \frac{1}{\sqrt{15}}
			\end{pmatrix}
			}_{v_3},\quad
			\underbrace{
			\begin{pmatrix}
				 -\frac{i}{2}-\frac{1}{2 \sqrt{15}} \\
				 \frac{1}{10} \left(5 i+\sqrt{15}\right) \\
				 \frac{1}{30} \left(-15 i+\sqrt{15}\right) \\
				 \frac{1}{\sqrt{15}}
			\end{pmatrix}
			}_{v_4}.
		\end{equation*}
		The matrices $U = (u_1|u_2|u_3|u_4)$ and $V = (v_1|v_2|v_3|v_4)$ satisfy
		\begin{equation*}
			|\det U| = \frac{2}{5 \sqrt{3}} = |\det V|
		\end{equation*}
		whence $T$ passes the Parallelepiped Test (Corollary \ref{CorollaryPT}).

		A further computation reveals that the matrices
		\begin{align*}
			U^*U &=\tiny
			\begin{pmatrix}
				 1 & -\frac{1}{3}-\frac{2 i}{3 \sqrt{5}} & \frac{4+\sqrt{3}}{6 \sqrt{2}}-\frac{i \left(1+4 \sqrt{3}\right)}{6 \sqrt{10}} 
				 & \frac{-4+\sqrt{3}}{6 \sqrt{2}}-\frac{i \left(-1+4 \sqrt{3}\right)}{6 \sqrt{10}} \\
				 -\frac{1}{3}+\frac{2 i}{3 \sqrt{5}} & 1 & \frac{-4+\sqrt{3}}{6 \sqrt{2}}+\frac{i \left(-1+4 \sqrt{3}\right)}{6 \sqrt{10}} 
				 & \frac{4+\sqrt{3}}{6 \sqrt{2}}+\frac{i \left(1+4 \sqrt{3}\right)}{6 \sqrt{10}} \\
				 \frac{4+\sqrt{3}}{6 \sqrt{2}}+\frac{i \left(1+4 \sqrt{3}\right)}{6 \sqrt{10}} 
				 & \frac{-4+\sqrt{3}}{6 \sqrt{2}}-\frac{i \left(-1+4 \sqrt{3}\right)}{6 \sqrt{10}} & 1 
				 & -\frac{1}{4}-\frac{3}{4} i \sqrt{\frac{3}{5}} \\
				 \frac{-4+\sqrt{3}}{6 \sqrt{2}}+\frac{i \left(-1+4 \sqrt{3}\right)}{6 \sqrt{10}} 
				 & \frac{4+\sqrt{3}}{6 \sqrt{2}}-\frac{i \left(1+4 \sqrt{3}\right)}{6 \sqrt{10}} 
				 & -\frac{1}{4}+\frac{3}{4} i \sqrt{\frac{3}{5}} & 1
			\end{pmatrix},
			\\
			V^*V&=
			\tiny
			\begin{pmatrix}
				 1 & \frac{1}{3}-\frac{2 i}{3 \sqrt{5}} & \frac{1}{2 \sqrt{3}}+\frac{i \left(3+4 \sqrt{3}\right)}{6 \sqrt{5}} 
				 & -\frac{1}{2 \sqrt{3}}-\frac{i \left(-3+4 \sqrt{3}\right)}{6 \sqrt{5}} \\
				 \frac{1}{3}+\frac{2 i}{3 \sqrt{5}} & 1 & -\frac{1}{2 \sqrt{3}}+\frac{i \left(-3+4 \sqrt{3}\right)}{6 \sqrt{5}} 
				 & \frac{1}{2 \sqrt{3}}-\frac{i \left(3+4 \sqrt{3}\right)}{6 \sqrt{5}} \\
				 \frac{1}{2 \sqrt{3}}-\frac{i \left(3+4 \sqrt{3}\right)}{6 \sqrt{5}} 
				 & -\frac{1}{2 \sqrt{3}}-\frac{i \left(-3+4 \sqrt{3}\right)}{6 \sqrt{5}} & 1 & -\frac{1}{2}+\frac{1}{2} i \sqrt{\frac{3}{5}} \\
				 -\frac{1}{2 \sqrt{3}}+\frac{i \left(-3+4 \sqrt{3}\right)}{6 \sqrt{5}} 
				 & \frac{1}{2 \sqrt{3}}+\frac{i \left(3+4 \sqrt{3}\right)}{6 \sqrt{5}} & -\frac{1}{2}-\frac{1}{2} i \sqrt{\frac{3}{5}} & 1
			\end{pmatrix},
		\end{align*}
		share the eigenvalues (given approximately by)
		\begin{equation*}
			2.73115,\quad 0.932497, \quad 0.253856, \quad 0.0824931.
		\end{equation*}
		Thus $T$ passes the Grammian Test (Corollary \ref{CorollaryGrammian}).  
		
		Recall that the $ij$th entries of $U^*U$ and $V^*V$ are $\inner{u_j,u_i}$
		and $\inner{v_j,v_i}$, respectively.  Therefore to check whether $T$ passes the Angle Test
		(Theorem \ref{TheoremWeak}), we need only compare the moduli of the entries
		of $U^*U$ and $V^*V$.  The moduli of the entries of $U^*U$ and $V^*V$ are equal, entry-by-entry,
		and given by
		\begin{equation*}\footnotesize
			\begin{pmatrix}
				 1 & \frac{1}{5} & \frac{2}{15} \left(3+\sqrt{3}\right) & \frac{2}{15} \left(3-\sqrt{3}\right) \\
				 \frac{1}{5} & 1 & \frac{2}{15} \left(3-\sqrt{3}\right) & \frac{2}{15} \left(3+\sqrt{3}\right) \\
				 \frac{2}{15} \left(3+\sqrt{3}\right) & \frac{2}{15} \left(3-\sqrt{3}\right) & 1 & \frac{2}{5} \\
				 \frac{2}{15} \left(3-\sqrt{3}\right) & \frac{2}{15} \left(3+\sqrt{3}\right) & \frac{2}{5} & 1
			\end{pmatrix}
		\end{equation*}
		Thus $T$ passes the Angle Test (Theorem \ref{TheoremWeak}).

		On the other hand, since
		\begin{equation*}
			\inner{u_1,u_2} \inner{u_2,u_3} \inner{u_3,u_1} = \tfrac{2}{75} (5-i \sqrt{5}) \neq 
			\tfrac{2}{75} (5+i \sqrt{5}) = \overline{ \inner{ v_1,v_2} \inner{v_2,v_3} \inner{v_3,v_1} },
		\end{equation*}
		the Strong Angle Test (Theorem \ref{TheoremStrong}) asserts that $T$ is \emph{not} UECSM.
		Similar computations reveal that the desired condition \eqref{eq-Cocycle} is violated for
		the triples $(i,j,k) = (1,2,4),(1,3,4),(2,3,4)$ as well.

		Working through the mechanics of the proof of Theorem \ref{TheoremStrong}, we find that the matrix
		$B = (\beta_{ij})$ is given by
		\begin{equation*}B = \tiny
			\begin{pmatrix}
				 1 & -1 & \frac{11+6 \sqrt{3}-2 i \sqrt{5}-3 i \sqrt{15}}{4 \sqrt{2} \left(3+\sqrt{3}\right)} 
				 & \frac{11-6 \sqrt{3}-2 i \sqrt{5}+3 i \sqrt{15}}{4 \sqrt{2} \left(-3+\sqrt{3}\right)} \\
				 -1 & 1 & \frac{11-6 \sqrt{3}+2 i \sqrt{5}-3 i \sqrt{15}}{4 \sqrt{2} \left(-3+\sqrt{3}\right)} 
				 & \frac{11+6 \sqrt{3}+2 i \sqrt{5}+3 i \sqrt{15}}{4 \sqrt{2} \left(3+\sqrt{3}\right)} \\
				 \frac{11+6 \sqrt{3}+2 i \sqrt{5}+3 i \sqrt{15}}{4 \sqrt{2} \left(3+\sqrt{3}\right)} 
				 & \frac{11-6 \sqrt{3}-2 i \sqrt{5}+3 i \sqrt{15}}{4 \sqrt{2} \left(-3+\sqrt{3}\right)} & 1 
				 & \frac{1}{8} \left(7-i \sqrt{15}\right) \\
				 \frac{11-6 \sqrt{3}+2 i \sqrt{5}-3 i \sqrt{15}}{4 \sqrt{2} \left(-3+\sqrt{3}\right)} 
				 & \frac{11+6 \sqrt{3}-2 i \sqrt{5}-3 i \sqrt{15}}{4 \sqrt{2} \left(3+\sqrt{3}\right)} 
				 & \frac{1}{8} \left(7+i \sqrt{15}\right) & 1
			\end{pmatrix}
		\end{equation*}
		In particular, each entry of $B$ is unimodular whence we once again see 
		that $T$ passes the Angle Test.  Also observe that the rank of $B$ is $4$ and its eigenvalues are approximately
		\begin{equation*}
		3.88114, \quad 0.694237, \quad -0.66798, \quad 0.0926015.
		\end{equation*}
		In particular, $B$ is neither rank-one nor positive.

		We should also mention that J.~Tener's procedure \UECSMTest{}
		also confirms, via entirely different methods (see Section \ref{SectionTener}),
		that $T$ is not UECSM.  
	\end{Example}

	The preceding example was discovered by the first author during a search of 10 million
	random integer matrices.  Such examples appear to be exceedingly rare and those
	which can be worked through in closed form rarer still.  Moreover, we were unable
	to find a $3 \times 3$ matrix with the same properties.

\section{Comparison with Tener's \UECSMTest}\label{SectionTener}

	J.~Tener's procedure \UECSMTest, introduced in \cite{Tener}, is an 
	effective tool in determining whether a given matrix is UECSM.  However, there
	are certain limitations inherent in the procedure.  To be more specific, \UECSMTest{}
	cannot be applied if the given matrix $T$ is $4 \times 4$ or larger \emph{and} either Cartesian component $A$ or $B$
	in the decomposition $T = A + iB$ (where $A = A^*$ and $B = B^*$) has a repeated eigenvalue. 

	On the other hand, the criterion for applying \Test{} is simply that the matrix $T$
	have distinct eigenvalues.  In this section, we compare the two procedures and demonstrate the existence of 
	matrices, both UECSM and not, for which either \UECSMTest{} or \Test{} (possibly both)
	fail to apply.  In particular, this demonstrates that \Test{} and \UECSMTest{} are complimentary
	procedures in the sense that neither test subsumes the other.

	Obviously many normal matrices (e.g., the $4 \times 4$ identity matrix) do not
	satisfy the hypotheses of either test.  This does not pose a problem, however, since
	the Spectral Theorem asserts that every normal
	matrix is unitarily equivalent to a \emph{diagonal matrix} whence every normal matrix is UECSM.
	In light of the preceding remarks, we therefore focus our attention on producing examples which are \emph{non-normal}.

	Finding non-normal matrices for which \Test{} is applicable and for which
	\UECSMTest{} is not is relatively straightforward.  Several examples are listed in Table 1 below
	(where $\sigma(T), \sigma(A), \sigma(B)$ denote the spectra of the operators $T,A,B$, respectively,
	in the decomposition $T = A+iB$, $A = A^*$, $B = B^*$).

	\begin{quote}\footnotesize
		\begin{equation*}
			\begin{array}{|c||c|c|c|c|c|}
			\hline
			T & \sigma(T) & \sigma(A) & \sigma(B) & \text{UECSM?} \\
			\hline
				\begin{pmatrix}
					 2 & 0 & 0 & 0 \\
					 0 & 4 & 0 & 0  \\
					 0 & 0 & 8 & 4 \\
					 0 & 0 & 0 & -2
				\end{pmatrix}
				& -2,2,4,8 & 2,4 ,3 \pm \sqrt{29} & 0,0,\pm 2 & \text{Yes} \\
				\hline
				\begin{pmatrix}
				 2 & 0 & 0 & 0 \\
				 0 & 0 & 4 & 0 \\
				 0 & 0 & 0 & 2 \\
				 0 & 8 & 0 & 0
				\end{pmatrix}
				& 2, 4, -2\pm 2i\sqrt{3}
				& \text{distinct}
				&0,0,\pm \sqrt{21}
				&\text{No} \\
				\hline
				\hline
				\begin{pmatrix}
				 4 & 1 & -1 & -2 \\
				 3 & 2 & -4 & 1 \\
				 -1 & -2 & 4 & 1 \\
				 -4 & 1 & 3 & 2
				\end{pmatrix}
				& -2,2,4,8 & 2,4, 3 \pm \sqrt{29} & 0,0,\pm 2 & \text{Yes} \\
				\hline
				\begin{pmatrix}
				 4 & -1 & 1 & -2 \\
				 -2 & 1 & -1 & 4 \\
				 -1 & 4 & -2 & 1 \\
				 1 & -2 & 4 & -1
				\end{pmatrix}
				& 2, 4, -2\pm 2i\sqrt{3}
				& \text{distinct}
				&0,0,\pm \sqrt{21}
				&\text{No} \\
				\hline
			\end{array}
		\end{equation*}
		\textsc{Table 1}:  Examples of simple matrices for which \Test{} is applicable and \UECSMTest{} is not.
		The third and fourth matrices listed are, respectively, unitarily equivalent to the first and second matrices.  The eigenvalues
		of the second and fourth matrices are distinct, but too long to display explicitly in the confines of the table.
	\end{quote}
	\smallskip

	In cases where $T$ has repeated eigenvalues,
	one frequently finds that both $A$ and $B$ both have distinct eigenvalues.  Such matrices are testable
	by \UECSMTest{} but not by \Test{}.  Several simple examples are listed in Table 2 below.
	
	\begin{quote}\footnotesize
		\begin{equation*}
			\begin{array}{|c|c|c|c|c|}
				\hline
				T & \sigma(T) & \sigma(A) & \sigma(B) & \text{UECSM?} \\
				\hline
				\begin{pmatrix}
				 0 & 18 & 0 \\
				 0 & 0 & 18 i  \\
				 0 & 0 & 0
				\end{pmatrix}
				& 0,0,0 & 0,\pm 9 \sqrt{2} & 0,\pm  9 \sqrt{2} & \text{Yes} \\
				\hline
				\begin{pmatrix}
					 0 & 18 & 0 \\
					 0 & 0 & 9 i  \\
					 0 & 0 & 0
				\end{pmatrix}
				& 0,0,0 & 0,\pm \frac{9\sqrt{5}}{2}  & 0,\pm  \frac{9\sqrt{5}}{2} & \text{No} \\
				\hline
				\hline
				\begin{pmatrix}
					 8+4 i & 4+8 i & -8+8 i \\
					 -8+2 i & -4+4 i & 8+4 i \\
					 4-4 i & 2-8 i & -4-8 i
				\end{pmatrix}
				& 0,0,0 & 0,\pm 9 \sqrt{2} & 0,\pm  9 \sqrt{2} & \text{Yes} \\
				\hline
				\begin{pmatrix}
					 8+2 i & 4+4 i & -8+4 i \\
					 -8+i & -4+2 i & 8+2 i \\
					 4-2 i & 2-4 i & -4-4 i
				\end{pmatrix}
				& 0,0,0 & 0,\pm \frac{9\sqrt{5}}{2}  & 0,\pm  \frac{9\sqrt{5}}{2} & \text{No} \\
				\hline
			\end{array}
		\end{equation*}
		\textsc{Table 2}:  Matrices for which \UECSMTest{} is applicable and \Test{} is not.
		The third and fourth matrices listed are, respectively, unitarily equivalent to the first and second matrices.
	\end{quote}
	\medskip

	It is possible to construct matrices for which neither \Test{}
	nor Tener's \UECSMTest{} is applicable.
	To be more specific, we exhibit several matrices $T$ such that
	\begin{enumerate}\addtolength{\itemsep}{0.5\baselineskip}
		\item $T$ has repeated eigenvalues (so that \Test{} is not applicable),
		\item $T = A+iB$ is at least $4 \times 4$ and either $A$ or $B$ has repeated eigevalues
			(so that \UECSMTest{} is not applicable).
	\end{enumerate}
	Although it is relatively straightforward to produce matrices $T$ satisfying (i) and (ii),
	it is naturally quite difficult to check whether $T$ is UECSM or not since by design 
	neither \Test{} nor \UECSMTest{} are applicable. 
	Fortunately, the set of matrices having properties (i) and (ii) has Lebesgue measure zero in $M_n(\C)$.
	
	We require a couple preliminary lemmas.  The following can be found in 
	\cite[Ex.~1]{Tener} or \cite[Ex.~1]{SNCSO}:

	\begin{Lemma}\label{Lemma3x3Tricky}
		The matrix
		\begin{equation}\label{eq-3x3NilpotentTrick}
			\megamatrix{0}{a}{0}{0}{0}{b}{0}{0}{0}
		\end{equation}
		is UECSM if and only if $ab = 0$ or $|a| = |b|$.
	\end{Lemma}
	
	In particular, the matrix \eqref{eq-3x3NilpotentTrick} is \emph{not} UECSM
	whenever $a$ and $b$ are nonzero and satisfy $|a| \neq |b|$.  In our construction,
	we intend to use \eqref{eq-3x3NilpotentTrick} as a building block  in conjunction with 
	the following lemma from \cite{CSPI}:
	
	\begin{Lemma}\label{LemmaZero}
		$T$ is UECSM if and only if $T \oplus 0$ is UECSM.
	\end{Lemma}
	
	In the lemma above, $T \oplus 0$ denotes the orthogonal direct sum of $T$ with a square zero matrix
	of any given size.  Since $T$ is UECSM if and only if $T - \lambda I$ is UECSM for any $\lambda \in \C$, it follows 
	from Lemma \ref{Lemma3x3Tricky} and Lemma \ref{LemmaZero} that the matrix
	\begin{equation}\label{eq-4x4AngleTenerExample}
		T=
		\left(
		\begin{array}{c|ccc}
		c & 0 & 0 & 0 \\
		\hline
		0 & 0 & a & 0 \\
		0 & 0 & 0 & b \\
		0 & 0 & 0 & 0 \\
		\end{array}
		\right)
	\end{equation}
	can be made UECSM or not according to our choice of $a$ and $b$ (the value of $c$ is irrelevant).
	A short computation then reveals that
	$\sigma(T) = \{0,0,0,c\}$ and
	\begin{align*}
		\sigma(A) &= \{0,\Re c,\pm \sqrt{|a|^2 + |b|^2} \},\\
		\sigma(B) &= \{0,\Im c,\pm \sqrt{|a|^2 + |b|^2}  \},
	\end{align*}
	whence if $c$ is either real or purely imaginary condition (ii) holds.  This leads us to the examples
	listed in Table 3 below:

	\begin{quote}\footnotesize
		\begin{equation*}
			\begin{array}{|c|c|c|c|c|}
				\hline
				T & \sigma(T) & \sigma(A) & \sigma(B) & \text{UECSM?} \\
				\hline
				\begin{pmatrix}
					4&0&0&0\\
					0&0&8&0\\
					0&0&0&8\\
					0&0&0&0\\
				\end{pmatrix}
				& 0,0,0,4 &	0,4,\pm \sqrt{2} & 0,0,\pm 4\sqrt{2} & \text{Yes} \\
				\hline
				\begin{pmatrix}
					8&0&0&0\\
					0&0&4&0\\
					0&0&0&8\\
					0&0&0&0\\
				\end{pmatrix}
				& 0,0,0,8 &	0,8,\pm 2\sqrt{5} &	0,0,\pm 2\sqrt{5} & \text{No} \\
				\hline
				\hline
				\begin{pmatrix}
					 5 & 1 & -3 & 1 \\
					 1 & -3 & 1 & 5 \\
					 1 & 5 & 1 & -3 \\
					 -3 & 1 & 5 & 1
				\end{pmatrix}
				& 0,0,0,4 &	0,4,\pm 4\sqrt{2} & 0,0,\pm 4\sqrt{2} & \text{Yes} \\
				\hline
				\begin{pmatrix}
					 5 & 1 & -1 & 3 \\
					 3 & -1 & 1 & 5 \\
					 1 & 5 & 3 & -1 \\
					 -1 & 3 & 5 & 1
				\end{pmatrix}
				& 0,0,0,8 &	0,8,\pm 2\sqrt{5} &	0,0,\pm 2\sqrt{5} & \text{No} \\
				\hline
			\end{array}
		\end{equation*}
		\textsc{Table 3}:  Matrices for which neither \UECSMTest{} nor \Test{} are applicable.
		The third and fourth matrices listed are, respectively, unitarily equivalent to the first and second matrices.
	\end{quote}


\begin{thebibliography}{99}

\bibitem{Chevrot}
	Chevrot, N., Fricain, E., Timotin, D.,
	\textit{The characteristic function of a complex symmetric contraction},
	Proc. Amer. Math. Soc. {\bf 135}~(2007), no. 9, 2877--2886
	MR2317964 (2008c:47025)
	
\bibitem{CRW}
	Cima, J.A., Ross, W.T., Wogen, W.R.,
	\textit{Truncated Toeplitz operators on finite dimensional spaces},
	Operators and Matrices, Vol.~2, no.~3, (2008), 357--369.

%\bibitem{Craven}
%  Craven, B.D.,
%  \textit{Complex symmetric matrices},
%  J. Austral. Math. Soc. {\bf 10}~(1969), 341--354.

%\bibitem{Gantmacher}
%  Gantmacher, F.R.,
%  \textit{The Theory of Matrices} (Vol. 2),
%  Chelsea, New York, 1989.
  
\bibitem{CCO}
	Garcia, S.R.,
	\textit{Conjugation and Clark operators},
	Contemp. Math., {\bf 393}~(2006) , 67--111.

\bibitem{CSOA}
	Garcia, S.R., Putinar, M.,
	\textit{Complex symmetric operators and applications},
	Trans. Amer. Math. Soc. {\bf 358}~(2006), 1285-1315.

\bibitem{CSO2}
	Garcia, S.R., Putinar, M.,
	\textit{Complex symmetric operators and applications II},
	Trans. Amer. Math. Soc.  {\bf 359}~(2007),  no. 8, 3913--3931

\bibitem{SNCSO}
	Garcia, S.R., Wogen, W.R.,
	\textit{Some new classes of complex symmetric operators},
	(to appear: Trans. Amer. Math. Soc.).
	\texttt{arXiv:0907.3761v1}

\bibitem{CSPI}
	Garcia, S.R., Wogen, W.R.,
	\textit{Complex symmetric partial isometries},
	J. Funct. Analysis {\bf 257}~(2009), 1251-1260.

\bibitem{Gilbreath}
	Gilbreath, T.M., Wogen, W.R.,
	\textit{Remarks on the structure of complex symmetric operators},
	Integral Equations Operator Theory {\bf 59}~(2007), no. 4, 585--590. 

\bibitem{HJ}
	Horn, R.A., Johnson, C.R., 
	\textit{Matrix Analysis}, 
	Cambridge Univ. Press, Cambridge, 1985.
	
\bibitem{Sarason}
	Sarason, D.,
	\textit{Algebraic properties of truncated Toeplitz operators},
	Oper. Matrices  {\bf 1}~(2007),  no. 4, 491--526.	
	
%\bibitem{ScottCanonical}
%  Scott, N.H.,
%  \textit{A new canonical form for complex symmetric matrices},
%  Proc. Roy. Soc. London Ser. A {\bf 441}~(1993), no. 1913, 625--640.

%\bibitem{Scott}
%  Scott, N.H.,
%  \textit{A theorem on isotropic null vectors and its application to thermoelasticity},
%  Proc. Roy. Soc. London Ser. A  {\bf 440}~(1993),  no. 1909, 431--442.	

\bibitem{Tener}
	Tener, J., \textit{Unitary equivalence to a complex symmetric matrix:  an algorithm},
	J. Math. Anal. Appl. {\bf 341}~(2008), no. 1, 640--648. 
	MR2394112
	
%\bibitem{TenerPC}
%	Tener, J., personal communication.
	
%\bibitem{Wellstein}
%  Wellstein, J.,
%  \textit{\"{U}ber symmetrische, alternierende und orthogonale Normalformen von Matrizen},
%  J. Reine Angew. Math. {\bf 163}~(1930), 166--182.	
	
\end{thebibliography}
\end{document}